\documentclass[a4paper,12pt,leqno]{amsart}

\usepackage{amssymb,amsfonts,amsmath,amsopn,amstext,amscd,latexsym,fullpage,verbatim}
\usepackage{amsthm,fancyhdr}
\usepackage[all]{xy}
\usepackage{pdfsync}

\newtheorem{Proposition}{Proposition}[section]
\newtheorem{Theorem}[Proposition]{Theorem}
\newtheorem{Corollary}[Proposition]{Corollary}
\newtheorem{Lemma}[Proposition]{Lemma}
\newtheorem{Example}[Proposition]{Example}

\newcommand{\spa}{\ }

\numberwithin{equation}{section}

\DeclareMathOperator{\GL}{GL}

\DeclareMathOperator{\Hom}{Hom}

\DeclareMathOperator{\Ind}{Ind}

\DeclareMathOperator{\Lie}{Lie}

\DeclareMathOperator{\Rep}{Rep}

\newcommand{\bbP}{{\mathbb P}}
\newcommand{\bbQ}{{\mathbb Q}}

\newcommand{\bB}{{\bf B}}

\newcommand{\bG}{{\bf G}}

\newcommand{\bL}{{\bf L}}

\newcommand{\bP}{{\bf P}}

\newcommand{\bT}{{\bf T}}
\newcommand{\bU}{{\bf U}}

\newcommand{\fra}{{\mathfrak a}}
\newcommand{\frb}{{\mathfrak b}}

\newcommand{\frd}{{\mathfrak d}}

\newcommand{\frg}{{\mathfrak g}}

\newcommand{\frl}{{\mathfrak l}}
\newcommand{\frn}{{\mathfrak n}}

\newcommand{\frp}{{\mathfrak p}}
\newcommand{\frq}{{\mathfrak q}}

\newcommand{\frt}{{\mathfrak t}}
\newcommand{\fru}{{\mathfrak u}}

\newcommand{\frx}{{\mathfrak x}}

\newcommand{\cD}{{\mathcal D}}

\newcommand{\cF}{{\mathcal F}}

\newcommand{\cL}{{\mathcal L}}
\newcommand{\cM}{{\mathcal M}}

\newcommand{\cO}{{\mathcal O}}

\newcommand{\cX}{{\mathcal X}}

\newcommand{\Qp}{\mathbb Q_p} 

\DeclareMathOperator{\alg}{alg}

\newcommand{\Oa}{\cO_{\alg}}

\begin{document}

\title[The JH-series of the locally analytic Steinberg representation]{The Jordan-H\"older series of the locally analytic Steinberg representation}

\author{Sascha Orlik, Benjamin Schraen}

\address{Fachgruppe Mathematik und Informatik \\ Bergische Universit\"at Wuppertal,
Gau\ss{}\-strasse 20 \\ 42119 Wuppertal\\ Germany.}
\email{orlik@math.uni-wuppertal.de}
\address{DMA \\ \'Ecole Normale Sup\'erieure, 45 rue d'Ulm \\ 75005 Paris
  \\ France.}
\email{benjamin.schraen@ens.fr}

\begin{abstract}
We determine the composition factors of a Jordan-H\"older series including multiplicities of the locally analytic Steinberg representation.
For this purpose we prove the acyclicity of the evaluated locally analytic Tits complex giving rise to the
Steinberg representation. Further we describe some analogue of the Jacquet functor applied to the irreducible principal
series representation constructed in \cite{OS2}.
\end{abstract}

\maketitle

\normalsize

\section{Introduction}

In this paper we study the locally analytic Steinberg representation $V^G_B$  for a given  split reductive $p$-adic Lie group $G$.
This type of object arises in various fields of Representation Theory, cf. \cite{Hum1,DOR}.
In the smooth representation theory of p-adic Lie groups, as well as in the case of finite groups of Lie type,
it is related to the (Bruhat-)Tits building and has therefore interesting applications \cite{Car,DM,Ca1}.
In the locally analytic setting it comes up so far in the $p$-adic Langlands program
with respect to semi-stable but not crystalline Galois representation \cite{Br}. More precisely,
$V^G_B$  coincides with the set of locally analytic vectors in the continuous Steinberg representation which should arise
in a possible local $p$-adic Langlands correspondence. Our main result
gives the composition factors including multiplicities of a Jordan-H\"older series for $V^G_B$. This answers a question raised by
Teitelbaum in \cite{T}.

Let $G$ be a split reductive $p$-adic Lie group over a finite extension $L$ of $\bbQ_p$ and let $B\subset G$ be a
Borel subgroup. The definition of $V^G_B$ is completely analogous to the above mentioned classical cases.
It is given by the quotient
$$V^G_B=I^G_B/\sum_{P \supsetneq B} I^G_P, $$
where $I^G_P=C^{an}(G/P,K)$ is the $G$-representation consisting of locally
$L$-analytic functions on the partial flag manifold $G/P$ with coefficients in some 
fixed finite extension  $K$ of $L$. In contrast to the smooth situation or that of a finite group of Lie type, the locally analytic
Steinberg representation is not irreducible. Indeed, it
contains the smooth Steinberg representation $v^G_B= i^G_B/\sum_{P
  \supsetneq B} i^G_P, $ where $i^G_P=C^\infty(G/P,K),$ as a closed
subspace.  On the other hand, $V^G_B$ has a composition series of finite
length and therefore the natural question of determining its Jordan-H\"older series comes up. Morita \cite{Mo} proved that for $G=\GL_2$, the topological dual of
$V^G_B$ is isomorphic to the space of $K$-valued sections of the canonical
sheaf $\omega$ on the Drinfeld half space $\cX=\bbP^1\setminus \bbP^1(L)$.  In
higher dimensions Schneider and Teitelbaum \cite{ST1} construct an injective
map from the space of $K$-valued sections of $\omega$ to the topological dual
of $V^G_B.$ However, the natural hope that this map is an isomorphism in
general turns out not to be correct. In fact, this follows by gluing some
results of \cite{Schr} and \cite{O} considering the Jordan-H\"older series of both representations
in the case of $\GL_3$.

In order to determine the composition factors of $V^G_B$ in the general case, we apply the
machinery  constructing locally analytic $G$-representations
$\cF^G_P(M,V)$ developed in \cite{OS2}. Here $M$ is an object of type $P$
in the category $\cO$ of Lie algebra  representations of $\Lie G$ and $V$ is a smooth
admissible representation of a Levi factor $L_P \subset P$ (we refer to
Section 2 for a more detailed recapitulation). It is proved in loc.cit. that
$\cF^G_P(M,V)$ is topologically irreducible if $M$ and
$V$ are simple objects and if furthermore $P$ is maximal for $M.$ On the
other hand, it is shown that $\cF^G_P$ is bi-exact which allows us to
speak of a locally analytic BGG-resolution.  This latter aspect is one main
ingredient for proving the acyclicity of the evaluated locally analytic Tits complex
$$ 0 \rightarrow  I^G_G \rightarrow \bigoplus_{K \subset \Delta \atop |\Delta\setminus K|=1}I^G_{P_K} \rightarrow \bigoplus_{K \subset \Delta \atop |\Delta\setminus K|=2}I^G_{P_K}
\rightarrow \dots \rightarrow \bigoplus_{K \subset \Delta
  \atop |K|=1}I^G_{P_K}\rightarrow I^G_B \rightarrow V^G_B\rightarrow 0.
$$
Here $\Delta$ is the set of simple roots with respect to $B$ and a choice of a maximal torus $T\subset B$.
Hence the determination of the composition factors of $V^G_B$ is reduced to the situation of an induced representation
$I^G_P$ which lies in the image of the functor $\cF^G_P.$ This leads to the question when two irreducible
representations of the shape $\cF^G_P(M,V)$ with $V$ a composition factor of $i^G_B$ are isomorphic (treated in Section 3).
It turns out that this holds true if and only if all ingredients are the same. Thus we arrive at the stage where Kazhdan-Lusztig theory
enters the game.

Now we formulate our main result.  For two reflections $w,w'$ in the Weyl
group $W$, let $m(w',w)\in {\mathbb Z}_{\geq 0}$ be the multiplicity of the
simple highest weight module $L(w \cdot 0)$ of weight $w \cdot 0\in
\Hom_L(\Lie T,K)$ in the Verma module $M(w' \cdot 0)$ of weight $w' \cdot 0.$
Let ${\rm supp}(w)\subset W$ be the subset of simple
reflections which appear in $w$.  Our main result is the following theorem in which we
identify the set of simple reflections with  the set $\Delta.$

\noindent {\bf Theorem:} {\it For $w \in W$, let $I \subset \Delta$ be a subset
such that the standard parabolic subgroup $P_I$ attached to $I$ is maximal for
$L(w \cdot 0)$. Let $v_{P_J}^{P_I}$ be the smooth generalized Steinberg
representation of $L_{P_I}$ with respect to a subset $J\subset I.$ Then the
multiplicity of the irreducible representation $\cF_{P_I}^G(L(w \cdot 0),
v_{P_J}^{P_I})$ in $V_B^G$ is given by
$$\sum_{w'\in W \atop {\rm supp}(w')=J} (-1)^{\ell(w')+|J|}m(w',w),$$ and we obtain in
this way all the Jordan-H\"older factors of $V_B^G$. In particular, the
smooth Steinberg representation is the only smooth subquotient of
$V_B^G$. Moreover, the representation $\cF_{P_I}^G(L(w \cdot 0),
v_{P_J}^{P_I})$ appears with a non-zero multiplicity if and only if $J
\subset {\rm supp}(w)$. }

\vspace{0.5cm}
We close this introduction by mentioning that we discuss in our paper more generally generalized locally analytic Steinberg  representation $V^G_P$, as well as their twisted versions $V^G_P(\lambda)$ involving a dominant weight $\lambda \in X^\ast(T)$.

\vspace{0.5cm} {\it Notation:} We denote by $p$ a prime and by $K \supset
L\supset \bbQ_p$ finite extensions of the field of $p$-adic integers $\Qp$.
Let $O_L$ be the ring of integers in $L$ and fix an uniformizer $\pi$
of $O_L$. We let $k_L=O_L/(\pi)$ be the corresponding residue field.
For a locally convex $K$-vector space $V$, we denote by $V'$ its strong dual, i.e.,
the $K$-vector space of continuous linear forms equipped with the strong topology of bounded convergence.

For an algebraic group ${\bf G}$  over $L$ we denote by $G={\bf G}(L)$ the p-adic Lie group of $L$-valued points.   We use
a Gothic letter $\frg$ to indicate its Lie algebra. We denote by $U(\frg):=U(\frg\otimes_L K)$
the universal enveloping algebra of $\frg$ after base change to $K.$  We let $\cO$ be the category $\cO$ of Bernstein-Gelfand-Gelfand in the sense of \cite{OS2}.

\section{The functors $\cF^G_P$}

In this first section we recall the definition of the functors $\cF^G_P$ constructed in \cite{OS2}. As explained in the introduction they are crucial for the determination of a Jordan-H\"older series of the
locally analytic Steinberg representation.

Let ${\bf G}$ be a split reductive algebraic group over $L$. Let ${\bf T}\subset {\bf G}$ be a maximal torus and fix a  Borel subgroup  ${\bf B\subset G}$ containing ${\bf T}$.  We identify the group $X^\ast({\bf T})$ of algebraic characters of $T$ via the derivative as a lattice in $\Hom_L(\frt,K)$.
Let ${\bf P}$ be standard parabolic subgroup (std psgp). Consider the Levi
decomposition ${\bf P=L_P \cdot U_P}$ where ${\bf L_P}$ is the Levi
subgroup containing ${\bf T}$ and ${\bf U_P}$ is its unipotent radical. Let ${\bf U_P^-}$ be
the unipotent radical of the opposite parabolic subgroup.   Let $\cO^\frp$ be the full subcategory of
$\cO$ consisting of $U(\frg)$-modules of type $\frp=\Lie{\bf P}$.  Its objects are $U(\frg)$-modules $M$
over the coefficient field $K$   satisfying the following properties:
\begin{enumerate}
\item The action of $\fru$ on $M$ is locally finite. \smallskip
\item The action of $\frl$ on $M$ is semi-simple and locally finite. \smallskip
\item $M$ is finitely generated as $U(\frg)$-module.
\end{enumerate}
In particular, we have $\cO=\cO^\frb.$ Moreover, if ${\bf Q}$ is another parabolic subgroup of ${\bf G}$ with ${\bf P} \subset {\bf Q}$, then
$\cO^\frq \subset \cO^\frp.$

Let ${\rm Irr}(\frl)^{\rm fd}$ be the set of finite-dimensional
irreducible $\frl$-modules. Any object in $\cO^\frp$ has by property (2) a direct sum decomposition into $\frl$-modules
\begin{equation}\label{isotypical_decom}
 M= \bigoplus_{\fra \in {\rm Irr}(\frl)^{\rm fd}} M_{\fra}
\end{equation}
where $M_{\fra}\subset M$ is the $\fra$-isotypic part of the finite-dimensional irreducible representation $\fra$.
We let $\Oa^\frp$ be the full subcategory of $\cO^\frp$
given by objects such that all $\frl$-representations  appearing in (\ref{isotypical_decom}) are induced by  finite-dimensional algebraic ${\bf L_P}$-representations.
The above inclusion $\cO^\frq \subset \cO^\frp$ is compatible with these new subcategories, i.e. we also have $\Oa^\frq \subset \Oa^\frp.$  
In particular, the category $\Oa^\frp$ contains all finite-dimensional $\frg$-modules which come from $G$-modules.
Every object in $\Oa^\frp$ has a Jordan-H\"older series which coincides with the Jordan-H\"older series in $\cO.$

Let $\Rep^{\infty,ad}_K(L_P)$ be the category of smooth admissible $L_P$-representations with coefficients over $K$. In \cite{OS2} there is constructed a bi-functor
$$\cF^G_P: \Oa^\frp\times \Rep^{\infty,ad}_K(L_P) \longrightarrow  \Rep_K^{\ell a}(G),$$
where $\Rep^{\ell a}_K(G)$ denotes the category of locally analytic  $G$-representations with coefficients in $K.$
It is contravariant in the first and covariant in the second variable. Furthermore, $\cF^G_P$ factorizes
through the full subcategory of admissible representations in the sense of Schneider and Teitelbaum \cite{ST2}.
Let us recall the definition of $\cF^G_P$.

Let $M$ be an object of $\Oa^\frp$. By the defining axioms (1) - (3) above  there is a surjective map
\begin{equation}\label{surjective_map}
\phi:U(\frg) \otimes_{U(\frp)} W \rightarrow M
\end{equation}
for some finite-dimensional  algebraic $P$-representation  $W\subset M$.
Let $V$ be a additionally a smooth admissible $L_P$-representation. We consider $V$ via the trivial action of $U_P$  as a $P$-representation.
Further by equipping $V$ with the finest locally convex topology it becomes a locally analytic $P$-representation, cf. \cite[\S$2$]{ST4}. Hence we may
consider the tensor product  representation $W' \otimes_K V$  as a locally analytic $P$-representation.
Let $\Ind^G_P : \Rep_K^{\ell a}(P) \to \Rep_K^{\ell a}(G) $ be the locally analytic induction functor \cite{Fe}.
Then
$$\cF^G_{P}(M,V) = \Ind^G_{P}(W' \otimes_K V)^{\frd}$$
denotes the subset of functions $f\in \Ind^G_{P}(W' \otimes _K V)$ which are killed by the ideal $\frd=\ker(\phi).$ 
It is shown in that $\cF^G_{P}(M,V)$  is a well-defined  $G$-stable closed subspace of
$\Ind^G_{P}(W'\otimes_K V)$ and has therefore a natural structure of a locally analytic $G$-representation.
Further the above construction is even functorial. If $V={\bf 1}$ is the trivial $L_P$-representation, we simply write
$\cF^G_P(M)$ instead of $\cF^G_P(M,{\bf 1}).$ The functors $\cF^G_P$ satisfy the following properties:

\begin{itemize}
\item (exactness) The bi-functor  $\cF^G_P$ is exact in both arguments.

\bigskip

\item (PQ-formula) Let $Q \supset P$ be another parabolic subgroup. We suppose that $L_P \subset L_Q.$
If $V$ is a smooth admissible $L_P$-representation, then we denote by
$$i^Q_P(V)=\Ind^{\infty,Q}_P(V)$$
the smooth induced $Q$-representation. Note that as $L_Q$-representation one has the identification
$i^Q_P(V)=i^{L_Q}_{L_P\cdot(U_P\cap L_Q)}(V).$ Let $M\in \Oa^\frq \subset \Oa^\frp.$
Then there is the following identity:
$$\cF^G_{P}(M,V)= \cF^G_{Q}(M,i^Q_P(V)).$$

\bigskip

\item (irreducibility)\footnote{Note that there is
  the restriction that $p>2$ if the root system of $G$ contains a factor of
  type $B,C,F_4$ resp.  $p>3$ if a factor of type $G_2$ occurs.} A std psgp $P\subset G$ is called maximal for an object $M\in \cO$ if $M \in \cO^\frp$ and if $P$ is maximal with
respect to this property. Let $P$ be a std psgp, maximal  for a simple object $M\in \Oa$. Further let $V$ be an  irreducible smooth admissible  $L_P$-representation,
then $\cF^G_P(M,V)$ is (topologically) irreducible.

\end{itemize}

In \cite{OS2} it is explained how one can deduce from the previous properties of the bi-functors $\cF^G_P$ the Jordan-H\"older series of a representation $\cF^G_P(M,V)$. Let us recall this procedure in the case $V={\bf 1}.$
The smooth generalized Steinberg representation to $P$ is the quotient
$$v^G_P= i^G_P/\sum_{P\subsetneq Q \subset G}i^G_Q.$$
This is an irreducible $G$-representation and all irreducible subquotients of $i^G_P$ occur in the shape $v^G_Q$
with $Q\supset P$ and with multiplicity one, cf. \cite[Ch. X,\S 4]{BoWa}, \cite{Ca2}.

Let $M$ be an object of the category $\Oa^\frp.$ Then it has a Jordan-H\"older series
$$M=M^0 \supset M^1 \supset M^2 \supset \ldots \supset M^r\supset M^{r+1}=(0) $$
in the same category. By applying the functor $\cF^G_P$ to it we get a sequence of surjections
$$\cF^G_P(M) \stackrel{p_0}{\to} \cF^G_P(M^1) \stackrel{p_1}{\to} \cF^G_P(M^2) \stackrel{p_2}{\to} \ldots \stackrel{p_{r-1}}{\to} \cF^G_P(M^r) \stackrel{p_r}{\to}   (0) .$$
For any integer $i$ with $0\leq i\leq r,$ we put
$$q_i:=p_i\circ p_{i-1} \circ \cdots \circ p_1 \circ p_0.$$
and set
$$\cF^i:=\ker(q_i)$$
which is a closed subrepresentation of $\cF^G_P(M).$
We obtain a filtration
\begin{equation}\label{filtration_neu}
\cF^{-1}=(0)\subset \cF^0 \subset \cdots \subset \cF^{r-1} \subset \cF^{r}=\cF^G_P(M)
\end{equation}
by closed subspaces with
$$\cF^i/\cF^{i-1}\simeq \cF^G_P(M^i/M^{i+1}) \spa .$$
Let $Q_i=L_i\cdot U_i\supset P$ a std psgp maximal for $M^i/M^{i+1}.$ Then
by the $PQ$-formula, we get
$$\cF^G_P(M^i/M^{i+1})=\cF^G_{Q_i}(M^i/M^{i+1},i^{Q_i}_P).$$
where $i^{Q_i}_P = i^{L_i}_{L_i\cap P}.$
We conclude that the representations
$$\cF^G_{Q_i}(M^i/M^{i+1},v^{Q_i}_{R})$$ where $R$ is a std psgp of $G$
with $Q_i \supset R \supset P$ and $v^{Q_i}_R = v^{L_i}_{L_i\cap R}$ are
the topologically irreducible constituents of
$\cF^G_{Q_i}(M^i/M^{i+1},i^{Q_i}_{P})$.  By refining the filtration
(\ref{filtration_neu}) we get thus a Jordan-H\"older series of $\cF^G_P(M)
$.

Finally we recall the parabolic BGG resolution of a finite-dimensional algebraic  $G$-representation \cite{Le}.
Let $W$ be the Weyl group of $G$ and consider the dot action $\cdot$  of $W$ on $X^\ast({\bf T})$  given by
$$w\cdot \chi= w(\chi+\rho)-\rho,$$
where $\rho=\frac{1}{2}\sum_{\alpha \in \Phi^+} \alpha.$
For a character $\lambda \in X^\ast({\bf T})$, let $M(\lambda)=U(\frg) \otimes_{U(\frb)} K_\lambda$ be the corresponding Verma module.
Clearly $M(\lambda)$ is an object of  $\Oa$.  We denote its irreducible quotient by $L(\lambda)$.
Let $\Delta$ be the set of simple roots and $\Phi$
the set of  roots of ${\bf G}$ with respect to a maximal torus  ${\bf T\subset B}$.
Let
$$X_+=\{\lambda \in  X^\ast({\bf T}) \mid (\lambda,\alpha^\vee )\geq 0 \; \forall \alpha \in \Delta \}$$
be the set of dominant weights in $X^\ast({\bf T}).$
If $\lambda \in X_+$, then $L(\lambda)$ is finite-dimensional and  comes from an irreducible algebraic $G$-representation.
In this situation, we also write $V(\lambda)$ for $L(\lambda).$

For a subset $I\subset \Delta$, let $P=P_I\subset G$ be the attached std psgp and denote by
$$X^+_I=\{\lambda \in X^\ast({\bf T}) \mid (\lambda,\alpha) \geq 0 \;\forall \alpha \in I\}$$
be the set of $L_I$-dominant weights.
Every $\lambda\in X^+_I$ gives rise to a  finite-dimensional algebraic  $L_I$-representation
\begin{equation}
V_I(\lambda)=V_P(\lambda).
 \end{equation}
We consider $V_I(\lambda)$ as a $P_I$-module by letting act $U_I$ trivially on it.
The generalized (parabolic) Verma module attached to the weight $\lambda$  is given by
$$M_I(\lambda)=U(\frg) \otimes_{U(\frp_I)} V_I(\lambda).$$
We have a surjective map
$$M(\lambda) \rightarrow M_I(\lambda),$$
where the kernel is given by the image of $\oplus_{\alpha \in I}
M(s_\alpha\cdot \lambda) \rightarrow M(\lambda)$, cf. \cite[Prop. $2.1$]{Le}.

Let $W_I\subset W$ be the parabolic subgroup induced by $I\subset \Delta.$ Consider the set $^IW=W_I\backslash W$
of left cosets which we identify with their representatives  of shortest length in $W$.
Let $^Iw$ be the element of maximal length in $^IW$.  If $\lambda$ is in $X_+$ and $w\in {}^IW$ then
$w\cdot \lambda \in X_I^+$, cf. \cite[p. 502]{Le}.
The $I$-parabolic BGG-resolution of $V(\lambda)$, $\lambda \in X_+$, is given by the exact sequence
\begin{multline*}
0 \rightarrow  M_I(^Iw\cdot \lambda) \rightarrow \bigoplus_{w\in{} ^IW\atop \ell(w)=\ell(^Iw)-1} M_I(w\cdot \lambda)
\rightarrow \dots \\ \cdots \rightarrow \bigoplus_{w\in{} ^IW \atop \ell(w)=1} M_I(w\cdot \lambda) \rightarrow M_I(\lambda) \rightarrow
V(\lambda)\rightarrow 0.
\end{multline*}
We refer to \cite{Le} resp. \cite{Ku} for the definition of the differentials in this complex.
By applying our exact functor $\cF^G_P$ to it, we get another exact sequence
\begin{multline*}
0 \leftarrow \cF^G_P( M_I(^Iw \cdot \lambda)) \leftarrow \bigoplus_{w\in ^IW\atop \ell(w)=\ell(^Iw)-1} \cF^G_P(M_I(w\cdot \lambda))
\leftarrow \dots \\
\cdots \leftarrow \bigoplus_{w\in ^IW \atop \ell(w)=1} \cF^G_P(M_I(w\cdot \lambda)) \leftarrow \cF^G_P (M_I(\lambda))
\leftarrow
\cF^G_P(V(\lambda))\leftarrow 0
\end{multline*}
which coincides by the very definition of $\cF^G_P$ and the PQ-formula with
\begin{multline}\label{parabolic_resolution}
0 \leftarrow \Ind^G_P(V_I(^Iw \cdot \lambda)') \leftarrow \bigoplus_{w\in{} ^IW\atop \ell(w)=\ell(^Iw)-1} \Ind^G_P(V_I(w\cdot \lambda)') \leftarrow \dots \\
\cdots \leftarrow \bigoplus_{w\in ^IW \atop \ell(w)=1} \Ind^G_P(V_I(w\cdot \lambda)') \leftarrow  \Ind^G_P(V_I(\lambda)')\leftarrow 
V(\lambda)'\otimes_K i^G_P \leftarrow 0.
\end{multline}

\section{Isomorphism between simple objects}

In this section we analyse when two simple modules of the shape
$\cF^G_{P_I}(M,v_{P_J}^{P_I})$, with $P_I$ maximal for $M$ and a subset $J \subset
I$, are isomorphic. This will be used in the next section for determining
the multiplicities of composition factors of the locally analytic Steinberg representation.

Let's begin by recalling some additional notation of \cite{OS2}. Let $\bG_0$ be a
split reductive group model of $\bG$ over $\mathcal{O}_L$. Let $\bT_0
\subset \bB_0 \subset \bG_0$ be $\cO_L$-models of $\bT$ and $\bB$. Fix a
std psgp $\bP_0$ and  denote by $\bU_{P,0}$ its unipotent radical. Let
$\bU_{P,0}^-$ be its opposite unipotent radical and denote by $\bL_0$ the Levi
factor containing $\bT_0$. Let $G_0=\bG_0(\cO_L)\subset G$, $P_0=\bP_0(\cO_L)=G_0
\cap P \subset P$ etc. be the corresponding  compact open subgroups consisting of $O_L$-valued points.  
We denote by $I=p^{-1}(\bG_0(k_L))\subset G_0$ the standard Iwahori subgroup 
where $p : \, G_0 \rightarrow \bG_0(k_L)$ is the reduction map.

Consider now  for an open subgroup $H$ of $G_0$  the distribution algebra $D(H):=D(H,K)=C_L^{an}(H,K)'$ which is defined by the
dual of the locally convex K-vector space $C_L^{an}(H,K)$ of locally $L$-analytic functions \cite{ST2}. It has the
structure of a Fr\'echet-Stein algebra. More precisely, for each
$\frac{1}{p} < r < 1$, there is a multiplicative norm $q_r$ on the Fr\'echet algebra $D(H)$ such that
if $D(H)_r$ denotes the Banach algebra given by the completion of $D(H)$ with respect to $q_r$, we have a topological
isomorphism of algebras $D(H) \simeq \varprojlim_r D(H)_r$. For the precise
definition of a Fr\'echet-Stein algebra we refer to loc.cit.
We set $P_H=P_0 \cap H$ and let $U(\frg,P_H)$ be the subalgebra of $D(H)$ generated by $U(\frg)$ and
$D(P_H)$. Let $U(\frg,P_H)_r$ be the topological closure of
$U(\frg,P_0)$ in $D(H)_r$.

The notion of a coadmissible module on a Fr\'echet-Stein algebra is defined in
\cite[\S$3$]{ST2}. If $\cM$ is such a coadmissible $D(H)$-module, it comes along
with a family $(q_{r,\cM})_r$ of seminorms such that if $\cM_r$ denotes the completion
of $\cM$ with respect to $q_{r,\cM}$, we have $\cM \simeq \varprojlim_r \cM_r$ and each
$\cM_r$ is a finitely-generated $D(H)_r$-module.

Let $M$ be a simple object of $\cO_{\alg}^{\frp}$  As such an object is has naturally the structure of a
$U(\frg,P_0)$-module. By \cite[Section 4]{OS2} there is a
continuous $D(G_0)$-isomorphism
\begin{equation*}
  D(G_0) \otimes_{U(\frg,P_0)} M  \simeq \cF_P^G(M)'.
\end{equation*}
Let  $\cM=D(H) \otimes_{U(\frg,P_H)} M$.  Then $\cM$ is a coadmissible $D(H)$-module in the above sense, cf.
\cite[Prop.4.4]{OS2} and $\cM_r$ is given by $\cM_r=D(H)_r \otimes_{U(\frg,P_H)} M$. In this special case the seminorm
$q_{r,\cM}$ is in fact a norm. We denote by $q_{r,M} $ its restrction to $M$ via the inclusion $M\hookrightarrow \cM$ and
by $M_r$ the completion of $M$ for the norm $q_{r,M}$, which we may identify with $U(\frg,P_H)_r\otimes_{U(\frg,P_H)} M.$
Thus we obtain a $U(\frg,P_H)_r$-equivariant map $M_r \rightarrow \cM_r$ giving rise to a $D(H)_r$-equivariant isomorphism
$$D(H)_r\otimes_{U(\frg,P_H)_r} M_r \xrightarrow{\sim} \cM_r .$$
Thus we get a topological isomorphism
\begin{equation}\label{isocoadm}
 \cM \simeq \varprojlim\nolimits_{r} D(H)_r \otimes_{U(\frg,P_H)_r} M_r.
\end{equation}

Let $\mathcal{D}$ be the category of topological $U(\frt)$-modules $M$
whose topology is metrizable, which are semi-simple with finite-dimensional
eigenspaces and such that the topology can be defined by a family of norms
$(q_r)_r$ such that
\begin{equation} \label{condnorm}
  q_r(\sum_\mu x_{\mu})=\sup_{\mu} q_r(x_{\mu}),
\end{equation}
for $x_{\mu} \in M_{\mu}$ in a decomposition $M=\bigoplus_{\mu \in
  \Hom_L(\frt,K)} M_{\mu}$. In this case, the completion $M_r$ of $M$ with respect to the
norm $q_r$ is given by
\begin{equation*}
  M_r=\{ \sum_\mu x_{\mu}\mid q_r(x_{\mu})  \rightarrow 0 \mbox{ cofinite} \}.
\end{equation*}
A simple consequence of this description is the uniqueness of the expansion $\sum_\mu
x_{\mu}$ and the fact that $(M_r)_{\mu}=M_{\mu}$ for all $\mu \in \Hom_L(\frt,K)$. In particular, the $U(\frt)$-eigenspaces
of $M_r$ are all finite-dimensional.

\begin{Lemma}
  As a subcategory of the category of $U(\frt)$-modules, the category $\cD$ is stable under subobjects and quotients and contains
  all the objects coming from $\cO_{alg}$.
\end{Lemma}

\begin{proof}
  A $U(\frt)$-submodule $S$ of an object $M$ in $\cD$ is clearly contained in $\cD$. In  particular, we have
  $S= \bigoplus_{\mu} S_{\mu}$ with $S_{\mu}=S \cap
  M_{\mu}$. Since the eigenspaces $M_{\nu}$ are finite-dimensional, each subspace $S_{\mu}$ is
  closed in $M_{\mu}$. It follows that $S$ is closed in $M$. As a consequence the
  quotient topology on $M/S$ is metrizable. It is given by the induced family of
  quotient norms $(\overline{q}_r)_r$. We have to check that condition
  \eqref{condnorm} is satisfied for each such norm $\overline{q}_r$. For every $\epsilon > 0$, there
  is an element $y=\sum_\mu y_\mu \in S$ such that
  \begin{align*}
    \overline{q}_r(\sum_\mu x_{\mu}+S) & \geq q_r(\sum_\mu x_{\mu}+y)-\epsilon \\
    & = \sup_\mu q_r(x_{\mu}+y_{\mu})- \epsilon \\
    & \geq \sup_\mu \overline{q}_r(x_{\mu} + S)-\epsilon.
  \end{align*}
 Hence $\overline{q}_r(\sum_\mu x_{\mu}+S) \geq \sup_\mu \overline{q}_r(x_{\mu} + S).$ The other inequality is immediate.

  As explained above, if $M$ is an object of $\cO_{\alg}$, then it has the structure of a topological
  metrizable $U(\frt)$-module. It is a consequence of \cite[Theorem $1.4.2$]{K} that each module of
  the form $U(\frg) \otimes_{U(\frb)} W \simeq U(\fru^-) \otimes_K W$, with
  $W$ a finite-dimensional representation of $B$, is an object of
  $\cD$.  But each object of $\cO_{\alg}$ is a quotient of some module of the shape  $U(\frg)
  \otimes_{U(\frb)} W$, so that we deduce the last assertion.
\end{proof}

The following statement generalizes \cite[Prop. 3.4.8]{OS1} in the split case and is again a consequence of \cite[1.3.12]{Fe}.

\begin{Proposition}\label{non-trivial intersection with U(s)-invariant
submodule}
Let $M$ be an object of $\cD$.

(i) We have an inclusion preserving bijection
\begin{eqnarray*}
\Big\{\mbox{closed $U(\frt)$-invariant subspaces of $M_r$} \Big\} &
\stackrel{\sim}{\longrightarrow} &
\Big\{\mbox{$U(\frt)$-invariant subspaces of $M$}\Big\}. \\
S & \longmapsto & S\cap M
\end{eqnarray*}
The inverse map is induced by taking the closure.

\medskip
(ii) Let $S \subset M_r$ be a closed $U(\frt)$-invariant
subspace, $S \neq 0$. Then $S \cap M \neq 0$. In
particular, any weight vector for the action of $\frt$ lies
already in $M$.
\end{Proposition}

\begin{proof}
 The proof follows by the above discussion and the semi-simplicity of $M_r.$
\end{proof}

Let $U=U_B$ be the unipotent radical of our fixed Borel subgroup $B\subset
G$ and let $\fru=\Lie U$ be its Lie algebra. Recall that if $N$ is a Lie algebra representation of $\frg,$
then $H^0(\fru,N)=\{n\in N \mid \frx\cdot n=0 \;  \forall\; \frx \in \fru\}$ denotes the subspace of vectors killed by
$\fru.$ This is a $U(\frt)$-module which is an object of $\cD$ if $N\in \cO_{\alg}.$

\begin{Corollary} \label{ncohomology} Let $M$ be an object of $\cO_{\alg}$.
Then $H^0(\fru,M_r)=H^0(\fru, M).$ In particular, $H^0(\fru,M_r)$  is finite-dimensional.
\end{Corollary}

\begin{proof}
We clearly have $H^0(\fru,M_r) \cap M = H^0(\fru,M).$ As $H^0(\fru,M_r)$ is closed
in $M_r$ by the continuity of the action of $\frg$  and as $H^0(\fru,M)$ is finite-dimensional and therefore complete the statement follows by Proposition \ref{non-trivial intersection with U(s)-invariant
submodule}.
\end{proof}

Let  $V$ be additionally a smooth admissible representation of $L_P$. Below we compute the
first $U$-homology resp. $U$-cohomology groups of the various representations
$\cF_P^G(M,V)$. More precisely, we denote by $\overline{H}_0(U,\cF_P^G(M,V))$ the
quotient of $\cF_P^G(M,V)$ by the topological closure of the $K$-subspace
generated by the elements $ux-x$ for $x \in \cF_P^G(M,V)$ and $u \in U$. It is the
largest Hausdorff quotient of $\cF_P^G(M,V)$ on which $U$ acts trivially.

\begin{Lemma}\label{Lemma_identity}
As Fr\'echet spaces we have $H^0(\frak{u}, \cF_P^G(M,V)') = H^0(\frak{u}, \cF_P^G(M)') \hat{\otimes}_K V'.$
\end{Lemma}

\begin{proof}
Since the action of $\fru$ is trivial on $V$, there is an identification $\cF^G_P(M,V)=\cF^G_P(M)\otimes_K V$
as $\fru$-modules. Now we write $V=\varinjlim_n V_n$ as the union of finite-dimensional $K$-vector spaces. Then
$\cF^G_P(M,V)=\varinjlim_n \cF^G_P(M)\otimes_K V_n$ as locally convex $K$-vector spaces. By passing to the dual we get
\begin{eqnarray*}
 \cF^G_P(M,V)' & = &  \varprojlim_n (\cF^G_P(M)\otimes_K V_n)' \\
 & = & \varprojlim_n \cF^G_P(M)'\otimes_K (V_n)' 
  =  \cF_P^G(M)'{} \hat{\otimes}_K V'.
\end{eqnarray*}
For the first identity confer \cite[Prop. 1.1.22]{Em} resp. \cite[Prop. 16.10]{NFA}, the second one follows as $V_n$ is
finite-dimensional, the third one is \cite[Prop. 1.1.29]{Em}. Now, the space $H^0(\fru, \cF^G_P(M)')$ is the kernel of the map 
$$d: \, \cF^G_P(M)' \rightarrow (\cF^G_P(M)')^{\dim \fru}$$
given by $v \mapsto (\frx_i v)$ where $(\frx_i)$ is a basis of $\fru$. By the exactness of the tensor product $- \otimes_K V_n$ and the left exactness 
of the projective limit, the space $H^0(\fru, \cF^G_P(M)') \hat{\otimes}_K V'$ is the kernel of the map
$d \hat{\otimes} 1 : \, \cF^G_P(M)' \hat{\otimes}_K V' \rightarrow (\cF^G_P(M)' \hat{\otimes}_K V')^{\dim \fru}$. The result follows.
\end{proof}

Now we can prove the main result of this section. Until the end of the section we will assume as in \cite{OS2} that  $p>2$ if the root system of $G$ contains a factor of
  type $B,C,F_4$ resp.  $p>3$ if a factor of type $G_2$ occurs.

\begin{Proposition}
    Let $M$ be a simple object of $\cO^\frp_{alg}$ such that $P$ is maximal
  for $M,$ and let $V$ be a smooth admissible representation of
  $L_P$. Let $\lambda \in X^\ast(T)$ be the highest weight of $M$, so that
  $M \simeq L(\lambda)$. Then there are $T$-equivariant isomorphisms
  \begin{equation}\label{first_isom}
    H^0(U, \cF_P^G(M,V)')=\lambda \otimes_K J_{U \cap L_P}(V)',
  \end{equation}
  and
    \begin{equation}
      \overline{H}_0(U,\cF_P^G(M,V))=\lambda' \otimes_K J_{U \cap L_P}(V),
  \end{equation}
  where $J_{U \cap L_P}$ is the usual Jacquet functor for the unipotent subgroup
  $U \cap L_P \subset L_P$.
\end{Proposition}

\begin{proof}
    The underlying topological space of $\cF_P^G(M,V)$ is of compact
  type. As the category of locally convex vector spaces of compact type is stable
  by forming Hausdorff quotients, the underlying topological vector space of $\overline{H}_0(U,
  \cF_P^G(M,V))$ is reflexive. As $H^0(U, \cF_P^G(M,V)')$ is the
  topological dual of $\overline{H}_0(U, \cF_P^G(M,V))$, it is sufficient
  to prove the first isomorphism (\ref{first_isom}).

  Let's begin by determining $H^0(\frak{u}, \cF_P^G(M,V)')$.  The Iwahori decomposition $G_0=\coprod_{w \in W^I} Iw P_0$ induces a
  decomposition
  \begin{equation*}
    D(G_0) \otimes_{U(\frg,P_0)} M  \simeq \bigoplus_{w \in W^I}
    D(I) \otimes_{U(\frg, I \cap wP_0w^{-1})} M^w.
  \end{equation*}
  Here $M^w$ denotes the module $M$ with the twisted action given by congugation with $w$. For each $w \in W^I$, we have
  \begin{multline*}
    H^0(\fru, D(I) \otimes_{U(\frg, I \cap wP_0w^{-1})} M^w)
    \simeq H^0(\mathrm{Ad}(w^{-1}) \fru, D(w^{-1}Iw) \otimes_{U(\frg,
      w^{-1}Iw \cap P_0)} \otimes M ).
  \end{multline*}
  Now we apply the precedent discussion with $H=w^{-1}Iw$. Set
  $\frn=\mathrm{Ad}(w^{-1})\fru$. 

  First we consider the case  $w \neq 1$. By \cite[Lemma
    $3.3.2$]{OS1}, there is an Iwahori decomposition $w^{-1}Iw=(U_{P,0}^-
    \cap w^{-1}Iw)(P_0 \cap w^{-1}Iw)$, hence there is by \cite[1.4.2]{K} a finite number of
    elements $u$ in $U^-_{P,0}$ such that $D(H)_r=\bigoplus \delta_u \cdot
    U(\frg, P_H)_r$, so that
  \begin{equation*} \cM_r \simeq \bigoplus_u \delta_u \otimes M_r
  \end{equation*}
  and the action of $\frx \in \frn$ is given by
  \begin{equation*}
    \frx \cdot \sum \delta_u \otimes m_u=\sum \delta_u \otimes
    (\mathrm{Ad}(u^{-1}) \frx )m_u.
  \end{equation*}
  Now we can find a non-trivial element $\frx \in \frak{u}_P^- \cap
  \mathrm{Ad}(w^{-1}) \frak{u}=\fru_P^- \cap \frn$. Indeed, the set $\Phi^- \cap
  w^{-1}(\Phi^+)$ contains an element of $\Phi^- \backslash \Phi_I^-$. For
  that, choose $\beta \in \Phi^+ \backslash \Phi_I^+$ such that
  $w^{-1}\beta \in \Phi^-$. This is possible since $w \notin W_I$. Then we
  have $w^{-1} \beta \notin \Phi_I^-$ since $W^I$ is exactly the set of $w$
  such that $w(\Phi_I^+) \subset \Phi^+$. Now $\mathrm{Ad}(u^{-1})\frx \in
  \fru_P^-$ since $u \in U_P^-$. By Corollary $8.5$ of \cite{OS2}, elements
  of $\frak{u}_P^-$ act injectively on $M$, and arguing as in Step $1$ of
  the proof of \cite[Theorem $5.1$]{OS2}, they  act injectively on
  $M_r$, as well. We conclude that $H^0(\mathrm{Ad}(u^{-1}) \frn, M_r)=0$ and therefore  $H^0(\frn, \cM_r)=0$. By
identity (\ref{isocoadm}), we
  get $H^0(\frn, \cM)=0$.

  Now consider the case $w=1$. Again we may write $D(I)_r=\bigoplus
  \delta_uU(\frg,P_0)_r$ for a finite number of $u \in U_{P,0}^-$, so that $D(I)_r
  \otimes_{U(\frg,P_0)_r} M_r=\bigoplus_u \delta_u \otimes M_r$. We will show
  that if $u \notin U_{P,0}^- \cap U(\frg,P_0)_r$, then
  $H^0(\mathrm{Ad}(u^{-1}) \fru, M_r)=0$. Here we will use Step
  $2$ in the proof of \cite[Theorem $5.1$]{OS2}. Let $\hat{M}$ be the
  formal completion of $M$, i.e. $\hat{M}=\prod_{\mu} M_{\mu}$ which is a
  $\frg$-module. The action of $\fru_P^-$ can be
  extended to an action of $U_P^-$ as explained in \cite[\S 5]{OS2}. If $\frx \in \frg$ and $u \in U_P^-$,
  the action of $\mathrm{ad}(u) \frx$ on $M_r$ is the restriction of the composite
  $u \circ \frx \circ u^{-1}$ on $\hat{M}$. As a consequence, we get
  $$H^0(\mathrm{ad}(u^{-1})\fru, M_r)=M_r \cap u^{-1} \cdot H^0(\fru, \hat{M}).$$ But
  $$H^0(\fru, \hat{M})=H^0(\fru, M)=Kv^+=K_\lambda$$ 
  where $v^+$ is a highest weight vector of $M$.  So if
  $H^0(\mathrm{ad}(u^{-1}) \fru, M_r) \neq 0$, we have $u^{-1}v^+ \in M_r$. By the proof of \cite[Theorem $5.1$]{OS2}, this
  gives a contradiction if $u \notin U_P^- \cap U_r(\frg, P_0)$. Hence by using identity (\ref{isocoadm}), we obtain finally an isomorphism
  \begin{equation*}H^0(\frak{u}, D(I) \otimes_{U(\frg,P_0)} M )
    \simeq H^0(\frak{u}, M).
  \end{equation*}
  By combining the result above together with Lemma \ref{Lemma_identity}, we get thus an isomorphism
  \begin{equation*}H^0(\frak{u}, \cF_{P}^G(M,V')) \simeq H^0(\frak{u}, M)
    \otimes_K V' \simeq K_\lambda
    \otimes_K V'\end{equation*}

  Now $H^0(U, \cF_P^G(M,V)')$ is a subspace of $H^0(\fru, \cF_P^G(M,V)')$ and this latter one is stable by the action of $U.$
  Thus we deduce that
  \begin{eqnarray*}
   H^0(U, \cF_P^G(M,V)') & = &  H^0(U,H^0(\fru, \cF_P^G(M,V)')) \\
	                 & = &  H^0(U, K_\lambda\otimes_K V') \\
			 & = & K_\lambda\otimes_K J_{U\cap L_P}(V)'.
  \end{eqnarray*}
\end{proof}

Next we formulate the main result of this section. For this recall that by Lemma $X.4.6$ of \cite{BoWa} the smooth induction $i^G_B$ has
Jordan-H\"older factors $v^G_{P_I}$ indexed by subsets $I \subset \Delta$
which appear all with multiplicity one. Moreover, by $X.3.2$, $(1)$ to
$(5)$, of \cite{BoWa}, if $Z$ is an irreducible subquotient of $i_B^G$,
then there is a smooth character $\delta$ of $T$ such that $Z$ is
the unique non-zero irreducible subrepresentation of $i_B^G(\delta)$ and
$\delta$ contributes to $J_U(Z)$.

\begin{Theorem}\label{isomorphic}
  Let $L(\lambda_1)$ and $L(\lambda_2)$ be two simple objects in the
  category $\mathcal{O}_{alg}$. For $i=1,2$, let $P_i$ be a std psgp
  maximal for $L(\lambda_i)$ and let $V_i$ be a simple
  subquotient of the smooth parabolic induction $i_B^{P_i}$. Then the two
  irreducible representations $\cF_{P_1}^G(L(\lambda_1), V_1)$ and
  $\cF_{P_2}^G(L(\lambda_2),V_2)$ are isomorphic if and only if $P_1=P_2$,
  $V_1=V_2$ and $\lambda_1=\lambda_2$.
\end{Theorem}

\begin{proof}
  Suppose that there is a non-trivial isomorphism between the irreducible
  representations $\cF_{P_1}^G(L(\lambda_1), V_1)$ and
  $\cF_{P_2}^G(L(\lambda_2),V_2)$. Let $\delta_2$ be a smooth character of
  $T$ such that $V_2\hookrightarrow i_B^{P_2}(\delta_2)$. By the sequence of
  embeddings
  \begin{equation*}
    \cF_{P_2}^G(L(\lambda_2),V_2) \subset \cF_{P_2}^G(L(\lambda_2), i_B^{P_2}(\delta_2)) \simeq \cF_B^G(L(\lambda_2), \delta_2) \subset
    \Ind_B^G(\lambda_2' \otimes_K \delta_2),
  \end{equation*}
  we obtain a non-trivial map $\cF_{P_1}^G(L(\lambda_1), V_1) \rightarrow
  \Ind_B^G(\lambda_2' \otimes \delta_2)$ and by Frobenius reciprocity a
  non-trivial $T$-equivariant map
  \begin{equation*}
    \overline{H}_0(U, \cF_{P_1}^G(M_1, V_1))=\lambda_1' \otimes_K J_{U \cap L_{P_1}}(V_1) \rightarrow
    \lambda_2' \otimes \delta_2.
  \end{equation*}
  It follows that $\lambda_1=\lambda_2$ and $P_1=P_2$. By \cite[X.3.2.(1)]{BoWa}, we know that $J_{U \cap L_{P_1}}(V_1)$ is a direct sum of
  smooth characters of $T$. By Frobenius reciprocity, these are exactly the
  smooth characters $\delta$ such that $V_1$ is an irreducible subobject of
  $i_B^{P_2}(\delta)$. As $\delta_2$ is one of them, we can conclude
  by \cite[X.3.2.(4)]{BoWa} that $V_1=V_2$.
\end{proof}

\section{The locally analytic Steinberg representation}

The section deals with the proof of our main theorem. Here we start with the proof of the acyclicity of the evaluated locally analytic
Tits complex.

As before let $P=P_I$ be a std psgp attached to a subset $I\subset \Delta$.  Let
$$I^G_P=I^G_P({\bf 1})=\Ind^G_{P}({\bf 1})$$ be the locally analytic $G$-representation induced from the trivial representations
${\bf 1}.$ If $Q\supset P$ is another parabolic subgroup, we get an injection $I^G_Q \hookrightarrow I^G_P.$
We denote by
$$V^G_P = I^G_P / \sum_{Q\supsetneq P} I^G_Q$$
the generalized locally analytic  Steinberg representation associated to $P$.
For $P=G,$ we have $V^G_G =I^G_G= {\bf 1}.$

The next result is the locally analytic analogue of the classical situations, cf. \cite[Ch.X,4]{BoWa}, \cite{Leh}, \cite{DOR}.

\begin{Theorem}
Let $I\subset \Delta$. Then the following complex is an acyclic resolution of $V^G_{P_I}$ by locally analytic $G$-representations,
\begin{equation}
0 \rightarrow  I^G_G \rightarrow \bigoplus_{I\subset K \subset \Delta \atop |\Delta\setminus K|=1}I^G_{P_K} \rightarrow \bigoplus_{I\subset K \subset \Delta \atop |\Delta\setminus K|=2}I^G_{P_K}
\rightarrow \dots \rightarrow \bigoplus_{I\subset K \subset \Delta
  \atop |K\setminus I|=1}I^G_{P_K}\rightarrow I^G_{P_I} \rightarrow V^G_{P_I}\rightarrow 0.
\end{equation}

\noindent {\rm Here the differentials $d_{K',K}:I^G_{P_{K'}} \rightarrow I^G_{P_K}$ are defined as follows.
For two subsets $K\subset K'$ of $ \Delta$, we
let
$$p_{K',K}:I^G_{P_{K'}} \longrightarrow I^G_{P_K}$$
be the natural homomorphism induced by the surjection $G/P_K \rightarrow G/P_{K'}.$
For arbitrary subsets  $K,K' \subset \Delta$, with $|K'|-|K|=1$
and $K'=\{k_1<\ldots <k_r\},$ we put
$$d_{K',K}=\left\{ \begin{matrix} (-1)^i p_{K',K} & K' =
K  \cup \{k_i\} \cr 0 & K \not\subset K' \end{matrix} \right. .$$}
\end{Theorem}

\vspace{0.5cm}

More generally, we will prove a variant of the above theorem. For this, let $\lambda\in X_+$ be a dominant weight.
For a std psgp $P=P_I\subset G$, we consider the finite-dimensional algebraic $P$-representation $V_I(\lambda)=V_P(\lambda)$
with highest weight $\lambda.$ We set
$$I^G_{P}(\lambda):=\Ind^G_{P}(V_P(\lambda)').$$
In particular, we get $I^G_G(\lambda)= V(\lambda)'.$ If $Q \subset G$ is
another parabolic subgroup with $P\subset Q$, then there is an injection
$$I^G_{Q}(\lambda) \hookrightarrow I^G_{P}(\lambda) $$
similarly as above for $V(\lambda)={\bf 1}.$ More precisely,  by the
transitivity of parabolic induction we deduce that $I_P^G(\lambda) \simeq
\Ind_Q^G(\Ind_P^{Q}(\lambda))$. As $V_Q(\lambda)'$ is a subobject of
$\Ind_P^{Q}(\lambda)$ we get the desired injection.  We define analogously as above 
the twisted generalized locally analytic Steinberg representation by
$$V^G_{P}(\lambda):= I^G_{P}(\lambda)/ \sum_{Q \supsetneq P} I^G_{Q}(\lambda).$$

\vspace{0.5cm}

\begin{Theorem}\label{resolution}
Let $\lambda \in X_+$ and let $I\subset \Delta$. Then the following complex is an acyclic resolution of $V^G_{P_I}(\lambda)$ by locally analytic $G$-representations,
\begin{multline}
0 \rightarrow  I^G_G(\lambda) \rightarrow \bigoplus_{I\subset K \subset \Delta \atop |\Delta\setminus K|=1}I^G_{P_K}(\lambda)
\rightarrow \bigoplus_{I\subset K \subset \Delta \atop |\Delta\setminus K|=2}I^G_{P_K}(\lambda)
\rightarrow \dots \\ \dots \rightarrow \bigoplus_{I\subset K \subset \Delta
  \atop |K\setminus I|=1}I^G_{P_K}(\lambda)\rightarrow I^G_{P_I}(\lambda) \rightarrow V^G_{P_I}(\lambda)\rightarrow 0.
\end{multline}
\end{Theorem}

\begin{proof}
The proof is by induction on the semi-simple rank ${\rm rk}_{\rm ss}(G)=|\Delta|$ of $G.$ We start with the case $|\Delta|=1.$ Then the complex (\ref{resolution}) coincides with
$$0\rightarrow I^G_G(\lambda) \rightarrow I^G_B(\lambda) \rightarrow V^G_B(\lambda)\rightarrow 0$$
and the claim is trivial.

Now, let $|\Delta|>1.$  We consider for any subset $K\subset \Delta$, the resolution $(\ref{parabolic_resolution}):$
\begin{multline*}
0 \leftarrow \Ind^G_{P_K}(V_K(w\cdot \lambda)') \leftarrow   \bigoplus_{w\in{}^KW \atop \ell(w)=\ell(w^K)-1}
\Ind^G_{P_K}(V_K(w\cdot \lambda)')    \leftarrow  \dots  \\ \dots \leftarrow \bigoplus_{w\in{}^KW \atop \ell(w)=1}
\Ind^G_{P_K}(V_K(w\cdot \lambda)') \leftarrow  I^G_{P_K}(\lambda) \leftarrow  i^G_{P_K}(\lambda ) \leftarrow 0.
\end{multline*}
Here we set $i^G_{P_K}(\lambda):=i^G_{P_K} \otimes_K V(\lambda)'.$
We abbreviate for any $w\in{}^KW$ and any integer $i\geq 0,$
$$I^G_{P_K}(w):=\Ind^G_{P_K}(V_K(w\cdot \lambda)'),$$

$$I^G_{P_K}[i]:=  \bigoplus\limits_{w\in{}^KW \atop \ell(w)=i} I^G_{P_K}(w).$$
and
$$V^G_{P_K}(w) = I^G_{P_K}(w) / \sum_{K'\supsetneq K \atop w \in{}^{K'}W} I^G_{P_{K'}}(w).$$


The complexes above induce hence a double complex

\vspace{0.5cm}

\label{bigdoublecomplex}
$
\begin{array}{cccccccc}
0 \rightarrow &   i^G_G(\lambda) &  \rightarrow &  \bigoplus\limits_{I\subset K \subset \Delta \atop |\Delta\setminus K|=1}i^G_{P_K}(\lambda) &
\rightarrow \dots \rightarrow & \bigoplus\limits_{I\subset K \subset \Delta
  \atop |K\setminus I|=1}i^G_{P_K}(\lambda) & \rightarrow & i^G_{P_I}(\lambda)  \\
& \parallel & & \downarrow &   &\downarrow &  &\downarrow  \\
0 \rightarrow &  I^G_G(\lambda) & \rightarrow & \bigoplus\limits_{I\subset K \subset \Delta \atop |\Delta\setminus K|=1}I^G_{P_K}(\lambda)  &
\rightarrow \dots \rightarrow & \bigoplus\limits_{I\subset K \subset \Delta
  \atop |K\setminus I|=1}I^G_{P_K}(\lambda) & \rightarrow & I^G_{P_I}(\lambda) \\
& \downarrow & & \downarrow &   &\downarrow &  &\downarrow \\
0 \rightarrow &  0 & \rightarrow & \bigoplus\limits_{I\subset K \subset \Delta \atop |\Delta\setminus K|=1} I^G_{P_K}[1] &
\rightarrow \dots \rightarrow & \bigoplus\limits_{I\subset K \subset \Delta
  \atop |K\setminus I|=1}  I^G_{P_K}[1] & \rightarrow &  I^G_{P_I}[1] \\
& \downarrow & & \downarrow &   &\downarrow &  &\downarrow  \\
0 \rightarrow &  0 & \rightarrow &  \bigoplus\limits_{I\subset K \subset \Delta \atop |\Delta\setminus K|=1} I^G_{P_K}[2] &
\rightarrow \dots \rightarrow & \bigoplus\limits_{I\subset K \subset \Delta
  \atop |K\setminus I|=1} I^G_{P_K}[2] & \rightarrow &   I^G_{P_I}[2] \\
& \downarrow & & \downarrow &   &\downarrow &  &\downarrow\\
&\vdots  & & \vdots  & &\vdots & & \vdots  \\
& \downarrow & & \downarrow &   &\downarrow &  &\downarrow  \\
0 \rightarrow &  0 & \rightarrow & 0  &
\rightarrow \dots \rightarrow & \bigoplus\limits_{I\subset K \subset \Delta
  \atop |K\setminus I|=1}  I^G_{P_K}[\ell(^Iw)-1] & \rightarrow &  I^G_{P_I}[\ell(^Iw)-1] \\
& \downarrow & & \downarrow &   &\downarrow &  &\downarrow  \\
0 \rightarrow &  0 & \rightarrow & 0 &
\rightarrow \dots \rightarrow & 0  & \rightarrow & I^G_{P_I}(^Iw) . \\
\end{array}
$

\vspace{0.5cm}

For proving our theorem, it suffices to show that each row, except from the second one, is exact apart from the very right hand side.
The upper line satisfies this property, since
it is tensor product of the (generalized) smooth Tits complex by the algebraic representation $V(\lambda)'.$ For $w\in{}^IW$, let $I(w)\subset \Delta$ be the unique maximal subset such that $w\in{}^{I(w)}W$ and set $P_w=P_{I(w)}.$
Hence it suffices to show the exactness of the evaluated sequence
$$ 0 \rightarrow    I^G_{P_w}(w) \rightarrow \dots \rightarrow  \bigoplus\limits_{I\subset K \subset I(w)
  \atop |K\setminus I|=1} I^G_{P_K}(w)  \rightarrow  I^G_{P_I}(w)  \rightarrow V^G_{P_I}(w)  \rightarrow 0 $$
for each $w\in{}^IW.$ This complex can be rewritten as follows. We have $I^G_{P_K}(w)= \Ind^G_{P_w}(I^{P_w}_{P_K}(w)).$ 
Since the induction functor $\Ind^G_{P_w}$ is exact, it suffices to show the exactness of the sequence
$$ 0 \rightarrow    I^{P_w}_{P_w}(w) \rightarrow \dots \rightarrow  \bigoplus\limits_{I\subset K \subset I(w)
\atop |K\setminus I|=1} I^{P_w}_{P_K}(w)  \rightarrow  I^{P_w}_{P_I}(w)  \rightarrow V^{P_w}_{P_I}(w)  \rightarrow 0 $$
where $V^{P_w}_{P_I}(w)=I^{P_w}_{P_I}(w) / \sum_{K\supsetneq I, w \in{}^KW} I^{P_w}_{P_K}(w).$
But we have $I^{P_w}_{P_K}(w)=I^{L_w}_{L_w\cap P_K}(w)$ and  ${\rm rk}_{\rm ss}(L_w) < {\rm rk}_{\rm ss} (G).$ Thus the complex above is by induction hypothesis exact. Hence the claim of our theorem follows.
\end{proof}

As a byproduct of the proof we get the following result (Note that $V^G_{P_I}(^Iw)= I^G_{P_I}(^Iw)$.):

\begin{Corollary}
For any dominant weight $\lambda\in \Lambda_+$ and any subset $I\subset \Delta$, there is an acyclic complex
$$0 \to v^G_{P_I}(\lambda) \to V^G_{P_I}(\lambda) \to \bigoplus\limits_{w\in{}^IW \atop \ell(w)=1} V^G_{P_I}(w) 
\to \cdots \to \bigoplus\limits_{w\in{}^IW \atop \ell(w)=\ell(^Iw)-1 } V^G_{P_I}(w) \to V^G_{P_I}(^Iw) \to 0 .$$
\end{Corollary}
\qed

\begin{Example}
Let $G=\GL_2$ and $\lambda=0.$ Then $W=\{1,s\}$ and the above complex is given by
$$0\to v^G_B \to V^G_B \to I^G_B(s)\to 0.$$
This complex was already considered by Morita \cite{Mo} resp.  Schneider and Teitelbaum \cite{ST5} and coincides
with the dual of the $K$-valued
``de Rham'' complex
$$0\to \cO(\cX)/K \to \Omega^1(\cX) \to H^1_{\rm dR}(\cX)\to 0$$ of the Drinfeld half space $\cX=\bbP^1\setminus \bbP^1(L)$.
\end{Example}

\vspace{0.5cm}

\begin{Lemma}
The representation $V^G_P(\lambda)$ has a Jordan-H\"older series of finite length.
\end{Lemma}

\begin{proof} The representation $V^G_P(\lambda)$ is a quotient of $I^G_P(\lambda)$ which coincides by definition of the functor $\cF^G_P$ with
$\cF^G_P(M_P(\lambda))$. As explained in Section 1, the latter object has a finite Jordan-H\"older series, hence
the same holds true for $V^G_P(\lambda).$
\end{proof}

From the previous lemma we see that a Jordan-H\"older series of
$V^G_B(\lambda)$ is induced by such a series of the induced representation
$I^G_P(\lambda)$. Now we are able to give a receipt for the determination of
the composition factors  with respect to $V^G_B.$
For two reflections $w,w'$ in $W$, let $m(w',w)\in {\mathbb Z}_{\geq 0}$ be the multiplicity of the
simple $U(\frak{g})$-module $L(w \cdot 0)$ in the Verma module $M(w' \cdot
0)$. It is known that $m(w',w)\geq 1$ if and only if $w'\leq w$ for the
Bruhat order $\leq$ on $W$.  These multiplicities numbers can be computed
using Kazhdan-Lusztig polynomials, as it was conjectured by Kazhdan and
Lusztig in \cite[Conj. $1.5$]{KaLu} and proved independently by
Beilinson-Bernstein \cite{BeBe} and Brylinski-Kashiwara
\cite{BrKa}. Let's recall that if $w \in W$ is an element of the Weyl
group, the support ${\rm supp}(w)$ of $w$ is the set of simple reflections
appearing in one (and so in any) reduced expression of $w$. In the
following we identify the set of simple reflections with $\Delta$.  Then $w \in W_I$ if and only if ${\rm supp}(w) \subset I$.

Recall from the proof of Theorem \ref{resolution}, that we defined for every $w\in W$ a subset $I(w)\subset \Delta$
with $w\in{}^{I(w)}W$ and that the subset $I(w)$ is maximal  with this property.  On the other hand,
by \cite[p. $186-187$]{Hum2}, the simple module $L(w \cdot \lambda)$ lies in $\cO^{\frp_I}$ if and
only if $w \in {}^IW$. It follows that for $w\in W$, the parabolic Lie Algebra $\frp_{I(w)}$ is maximal for
$L(w\cdot \lambda).$

\begin{Theorem}\footnote{Note that as before there is
  the restriction that $p>2$ if the root system of $G$ contains a factor of
  type $B,C,F_4$ resp.  $p>3$ if a factor of type $G_2$ occurs.}\label{multiplicity}
  Fix $w \in W$ and let $I=I(w) \subset \Delta$ be as above. For a subset $J\subset \Delta$ with $J\subset I$, let $v_{P_J}^{P_I}$ be the generalized 
smooth Steinberg  representation of $L_{P_I}$. Then the multiplicity of the irreducible $G$-representation
  $\cF_{P_I}^G(L(w \cdot \lambda), v_{P_J}^{P_I})$ in $V_B^G(\lambda)$ is
  \begin{equation}\label{alternating_sum}
   \sum_{w'\in W \atop {\rm supp}(w')=J} (-1)^{\ell(w')+|J|}m(w',w),
  \end{equation}
and we obtain
  in this way all the Jordan-H\"older factors of $V_B^G(\lambda)$. Moreover, this
  multiplicity is non-zero if and only if $J \subset {\rm supp}(w)$.  In
  particular, the locally algebraic representation $v^G_B(\lambda)$ is the
  only locally algebraic subquotient of $V_B^G(\lambda)$. 
\end{Theorem}

\begin{proof}
  By Theorem \ref{resolution}, we get the
  following formula for the multiplicity of the simple object $\cF_{P_I}^G(L(w \cdot
  \lambda), v_{P_J}^{P_I})$ in $V^G_B(\lambda)$:
  \begin{equation*}
    [V_B^G(\lambda) : \cF_{P_I}^G(L(w \cdot \lambda), v_{P_J}^{P_I})]=\sum_{K \subset \Delta}
    (-1)^{|K|} [I_{P_K}^G(\lambda) : \cF_{P_I}^G(L(w \cdot \lambda), v_{P_J}^{P_I})].
  \end{equation*}
  By definition we have $I_{P_K}^G(\lambda)=\cF^G_{P_K}(M_K(\lambda))$. Since the Jordan-H\"older series
  of $M_K(\lambda)$ lies in $\Oa^{\frp_K}$, as well, we deduce by
  by  Theorem \ref{isomorphic}  and the discussion in Section 2 that $K\subset I$ is a necessary condition for
  the non-vanishing of the expression $[I_{P_K}^G(\lambda) : \cF_{P_I}^G(L(w \cdot \lambda),
  v_{P_J}^{P_I})]$.  On the other hand, $v_{P_J}^{P_I}$ is
  a subquotient of $i_{P_K}^{P_I}$ if and only if $K \subset J$. Again by
  Theorem \ref{isomorphic} and Section 2, the term $\cF_{P_I}^G(L(w \cdot \lambda),
  v_{P_J}^{P_I})$ appears in $I_{P_K}^G(\lambda)$ only if $K \subset
  J$.

So, let $K \subset J$ and consider the JH-component $Q=L(w \cdot \lambda)$ of $M_K(\lambda)$ in $\cO_{\alg}.$
Then  $\cF_{P_K}^G(Q,1)=\cF_{P_I}^G(Q, i_{P_K}^{P_I})$ by the PQ-formula. In a Jordan-H\"older series of $\cF_{P_I}^G(Q,
  i_{P_K}^{P_I})$ the term $\cF_{P_I}^G(Q, v_{P_J}^{P_I})$  appears
  with multiplicity one by Theorem \ref{isomorphic}, so that
  \begin{equation*}
    [I_{P_K}^G(\lambda) : \cF_{P_I}^G(L(w \cdot \lambda),
    v_{P_J}^{P_I})]=[M_K(\lambda) : L(w \cdot \lambda)].
  \end{equation*}
  As a consequence, we have
  \begin{equation*}
    [V_B^G(\lambda) : \cF_{P_I}^G(L(w \cdot \lambda), v_{P_J}^{P_I})]=\sum_{K \subset J}
    (-1)^{|K|} [M_K(\lambda) : L(w \cdot \lambda)].
  \end{equation*}
  Now we make use of the character formula  
$${\rm ch} M_K(\lambda)= \sum_{w' \in W_K}
    (-1)^{\ell(w')} {\rm ch} M(w' \cdot \lambda),$$
cf. \cite[Proposition $9.6$]{Hum2}. We obtain
  \begin{align*}
    [V_B^G(\lambda) : \cF_{P_I}^G(L(w \cdot \lambda),
    v_{P_J}^{P_I})]&=\sum_{K \subset J} (-1)^{|K|} \sum_{w' \in W_K}
    (-1)^{\ell(w')} [M(w' \cdot \lambda) : L(w \cdot
    \lambda)] \\
    &= \sum_{w' \in W} (-1)^{\ell(w')} [M(w' \cdot \lambda) : L(w \cdot
    \lambda)] \sum_{{\rm supp}(w') \subset K \subset J} (-1)^{|K|}.
  \end{align*}
  But the sum $$\sum_{{\rm supp}(w') \subset K \subset J}
  (-1)^{|K|}=(1-1)^{|J \backslash {\rm supp}(w')|}$$ is non-zero if and only
  if ${\rm supp}(w')=J$, so we obtain the formula. The
  non-vanishing criterion follows from Corollary \ref{inequality} in
  the appendix.
\end{proof}

\begin{Example}
  a) Let $G=\GL_3,$ $\Delta=\{\alpha_1, \alpha_2 \}$. Let $s_i$ be the
  element of $W$ corresponding to $\alpha_i$ and abbreviate
  $P_i=P_{\{ \alpha_i \} }$. In this case, $m(w',w)=1$ if and only if $w' \leq w$
  for the Bruhat order (cf. \cite[8.3.(c)]{Hum2}). As a consequence, we
  obtain the following Jordan-H\"older factors all with multiplicity one :
 \begin{gather*}
   v_B^G \\ \cF_{P_1}^G(L(s_2 \cdot 0), v_B^{P_1}), \, \cF_{P_2}^G(L(s_1
   \cdot 0), v_B^{P_2}) \\ \cF_{P_1}^G(L(s_2s_1 \cdot 0), 1), \,
   \cF_{P_1}^G(L(s_2s_1 \cdot 0),
   v_B^{P_1}) \\
   \cF_{P_2}^G(L(s_1s_2 \cdot 0), 1), \, \cF_{P_2}^G(L(s_1s_2 \cdot 0), v_B^{P_2}) \\
   \cF_B^G(L(s_1s_2s_1 \cdot 0), 1).
 \end{gather*}
This particular result was already shown in \cite{Schr}.

b) Let $G=\GL_4$. Here there are exactly two subquotients of $V_B^G$
with multiplicity greater than one. We fix notation as follows. Let
$\Delta=\{\alpha_1, \alpha_2, \alpha_3 \}$ such that $s_1$ and $s_3$
commute. We use the more compact notation $P_i=P_{\{\alpha_i \}}$,
$P_{i,j}=P_{\{\alpha_i, \alpha_j \}}$, $W^{i,j}=W^{\{\alpha_i,\alpha_j\}}$ etc. The two subquotients with
multiplicity greater than one are then $\cF_{P_{1,3}}^G(L(s_2s_3s_1s_2
\cdot 0), v_B^{P_{1,3}})$ and $\cF_{P_2}^G(L(s_1s_3s_2s_3s_1 \cdot 0),
v_B^{P_2})$ and they have both multiplicity $2$. The length of $V_B^G$ is
$50$ and it contains $48$ non-isomorphic Jordan-H\"older components. Let
us remark that this example also shows that Kazhdan-Lusztig multiplicities are the
reason for unexpected components in $V_B^G$. For instance, the component
$\cF_{P_2}^G(L(s_1s_3s_2s_1s_3 \cdot 0, 1))$ appears in
$I_{P_2}^G$, but passing to $V_B^G$ kills only one of its two occurrences in
$I_B^G$.

Here is the idea to carry out this computation. The number of components of
the form $\cF_{P_I}^G(L(w \cdot \lambda, V)$ for a fixed $w \in W^I$ with
$I=I(w)$ maximal and $V$ a smooth irreducible subquotient of $i^{L_I}_{L_I\cap B}$, is
exactly the number of subsets of $I \cap {\rm supp}(w)$, that is $2^{|I
  \cap {\rm supp}(w)|}$. Let $W^I_p$ be the subset of $W^I$ consisting of
elements for which $I$ is maximal, an easy computation shows that
 \begin{equation}
   |W^I_p|=\sum_{J \subset \Delta \backslash I} (-1)^{|J|}|W^{I \cup
     J}|=\sum_{J \subset \Delta \backslash I} (-1)^{|J|}|W|/|W_{I \cup J}|.
 \end{equation}
 If $I=\varnothing$ or $I=\Delta$, then $|W^I_p|=1$ and $|I\cap {\rm supp}(w)|=1.$ Furthermore, we compute easily that
 \begin{equation*}
   \begin{gathered} |W^{1,2}_p|=|W^{2,3}_p|=|W^{1}_p|=|W^{3}_p|=3 \\
     |W^{1,3}_p|=|W^{2}_p|=5.
   \end{gathered}
 \end{equation*}
 We see for example that $W^{1,2}_p=\{s_3, s_3s_2,s_3s_2s_1 \}$ so
 that $|\{1,2 \} \cap {\rm supp}(w)|$ takes the values $0$, $1$ and $2$ when $w$
 is varying in $W^{\{1,2 \}}_p$. This procedure gives us the cardinality of the composition factors appearing in
 $V_B^G$. As for the multiplicities, we can use a software package like
 CHEVIE (\cite{GHLMP}) to see that the only simple modules $L(w \cdot 0)$ having
 multiplicities bigger than one in $M(0)$ are given by $w=s_2s_1s_3s_2$ and
 $w=s_1s_3s_2s_1s_3$.
\end{Example}

\section{Appendix}

In this appendix, we prove the last assertion of Theorem \ref{multiplicity}
by interpreting the alternate sum (\ref{alternating_sum}) as the multiplicity of $L(w \cdot
\lambda)$ in the $H_0$ of a subcomplex of the BGG complex. Then we show that
the submodule of $M(\lambda)$ generated by the $M(w \cdot \lambda)$ with $w$ a
Coxeter element in $W_J$ is a quotient of this $H_0$. We fix the same notation as before.

Let $\lambda \in X_+$ and recall from Section 1 that the BGG
resolution of $L(\lambda)=V(\lambda)$ takes the form
\begin{equation*}
  \cdots \rightarrow C_i(\frg) \xrightarrow{\epsilon_i} C_{i-1}(\frg)
  \xrightarrow{\epsilon_{i-1}} \cdots \xrightarrow{\epsilon_1} C_0(\frg)
  \rightarrow L(\lambda) \rightarrow 0
\end{equation*}
with
$$C_i(\frg)=\bigoplus_{\ell(w)=i} M(w \cdot \lambda).$$ For each $w \in W$, fix a
non zero injection $M(w \cdot \lambda) \hookrightarrow M(\lambda)$. This
gives, for each $w' \leq w$, a unique compatible injection $i_{w,w'} : \,
M(w \cdot \lambda) \hookrightarrow M(w' \cdot \lambda)$.

Let $I$ be a subset of $\Delta$. For $i \geq 0$, let $C_i(\frg, I) \subset
C_i(\frg)$ be defined by
$$ C_i(\frg, I) := \bigoplus_{w \in W_I, \ell(w)=i} M(w \cdot \lambda).$$
As $\epsilon_i(C_i(\frg,I)) \subset C_{i-1}(\frg, I)$, we obtain
a subcomplex $C_{\bullet}(\frg, I)$ of $C_{\bullet}(\frg)$. Now we abbreviate $W_I^{(i)}$ for the
subset of $W_I$ consisting of elements of length $i$.

\begin{Lemma}\label{tensor}
  The complex $C_{\bullet}(\frg, I)$ is isomorphic to the complex $U(\frg)
  \otimes_{U(\frp_I)} C_{\bullet}(\frl_I)$, hence it is exact in positive
  degrees.
\end{Lemma}

\begin{proof}
  We can write $\epsilon_i|_{C_i(\frg,I)}=\sum_{w_1,w_2} a_{w_1,w_2}^i
  i_{w_1,w_2}$ where the sum is over all $w_1 \in W_I^{(i)}$ and $w_2 \in
  W_I^{(i-1)}$ such that $w_2 \leq w_1$. By
  \cite[Lemma $10.2$ ]{Ro} we have $a_{w_1,w_2}^i \neq 0$ for each such pair
  $(w_1,w_2)$. We fix injections $U(\frl_I) \otimes_{U(\frb \cap
  \frl_I)} K_{w \cdot \lambda} \hookrightarrow U(\frl_I) \otimes_{U(\frb \cap
  \frl_I)} K_\lambda$, $w\in W_I$, giving rise to well defined maps 
$$i_{w,w'}^I : \, U(\frl_I)  \otimes_{U(\frb \cap \frl_I)} K_{w \cdot \lambda} \hookrightarrow U(\frl_I)
  \otimes_{U(\frb \cap \frl_I)} K_{w' \cdot \lambda}$$ such that after
  applying $U(\frg) \otimes_{U(\frp_I)} -$, we have the equality ${\rm id} \otimes
  i_{w,w'}^I=i_{w,w'}$. Now consider the complex
  $\widetilde{C}_{\bullet}(\frl_I)$ where
  $\widetilde{C}_{\bullet}(\frl_I)=C_{\bullet}(\frl_I)$ and where the
  differential maps are given by $\tilde{\epsilon}_i=\sum_{w_1 \leq w_2 \in
  W_I}
  a_{w_1,w_2}^i i_{w_1,w_2}^I$, so that $C_{\bullet}(\frg,I)=U(\frg)
  \otimes_{U(\frp_I)} \widetilde{C}_{\bullet}(\frl_I)$. As the Bruhat order
  on $W_I$ is induced by the Bruhat order of $W$, the squares (see
  \cite[Definition $10.3$]{Ro}) of $W$ with elements in $W_I$ are
  exactly the squares of $W_I$, so that $\widetilde{C}_{\bullet}(\frl_I)$
  is a complex. By Corollary $10.7$ or Lemma $10.5$ of \cite{Ro} the
  complex $\widetilde{C}_{\bullet}(\frl_I)$ is exact in positive degrees
  and even isomorphic to the BGG complex $C_{\bullet}(\frl_I)$.
\end{proof}

For every integer $i \geq 0$, let $C_{\bullet}^{\leq i}(\frg,I)$ be the subcomplex
of $C_{\bullet}(\frg,I)$ defined by
$$C_j^{\leq i}(\frg,I)=\bigoplus_{\exists J\subset I, |J| \leq i \atop w\in W_J, \ell(w)=j} M(w \cdot \lambda)$$ and
set $\overline{C}_{\bullet}(\frg,I)=C_{\bullet}(\frg,I)/C_{\bullet}^{ \leq |I|-1}(\frg,I)$.

\begin{Proposition}
  The complex $C_{\bullet}^{\leq i}(\frg,I)$ is exact in degrees $\geq i+1$
  and the complex $\overline{C}_{\bullet}(\frg,I)$ in degree $\geq |I|+1$.
\end{Proposition}

\begin{proof}
  We will prove this proposition by induction on $i$. For $i=0$, this is
  clear. Now we can apply induction hypothesis to the long exact sequence coming from
  the exact sequence of complexes
  \begin{equation} \label{exeqcomplex}
    0 \rightarrow C_{\bullet}^{\leq i-1}(\frg,I) \rightarrow
    C_{\bullet}^{\leq i}(\frg,I) \rightarrow \bigoplus_{J \subset I, |J|=i}
    \overline{C}_{\bullet}(\frg,J) \rightarrow 0.
  \end{equation}
  (Here we use the identity $W_{J_1} \cap W_{J_2}=W_{J_1 \cap J_2}$ for the appearance of the direct sum.)
  In the special case, where $|I|=i$, we can deduce from Lemma \ref{tensor} that
  $C_{\bullet}(\frg,I)$ is acyclic in strictly positive degrees, hence by
  induction that $H_j(\overline{C}_{\bullet}(\frg,I))=0$ for $j > i$. In general
  we can use \eqref{exeqcomplex} to deduce that
  $H_j(C_{\bullet}^{\leq i}(\frg,I))=0$ when $j > i$.
\end{proof}

\begin{Corollary} \label{inequality}
  For $I \subset \Delta$,  let $M_I$ be the submodule of $M(\lambda)$
  generated by the Verma modules $M(w \cdot \lambda)$ where  $w$ is a Coxeter
  element in $W_I$. For $w \in W$, we have
  \begin{equation}
    [M_I:L(w \cdot \lambda)] \leq \sum_{{\rm supp}(w')=I} (-1)^{|I|+\ell(w')}[M(w'
    \cdot \lambda):L(w \cdot \lambda)].
    \end{equation}
  Hence this sum is non-zero if and only if $I \subset {\rm supp}(w)$.
\end{Corollary}

\begin{proof}
  As the complex $\overline{C}_{\bullet}(\frg,I)$ is exact in degree $>
  |I|$ and $\overline{C}_j(\frg,I)=0$ for $j < |I|$, it is sufficient to
  prove that $M_I$ is a quotient of
  $H_{|I|}(\overline{C}_{\bullet}(\frg,I))$. Remark that
  $M'=\epsilon_{|I|}(\bigoplus_{{\rm supp}(w)=I, \ell(w)=|I|}M(w))$ is a
  quotient of $H_{|I|}(\overline{C}_i(\frg,I))$. Then the image of $M'$ by
  $\sum_{{\rm supp}(w)=I, \ell(w)=|I|} i_{w,1}$ is exactly $M_I$, and we
  deduce the result. It is not difficult to see that there exists a Coxeter
  element $w'$ of $W_I$ such that $w' \leq w$ if and only if $I \subset
  {\rm supp}(w)$. 
\end{proof}

\end{document}